\documentclass[12pt,twoside,letter]{amsart}
\usepackage{mathrsfs}
\usepackage{mathrsfs}
\usepackage{txfonts}
\usepackage{bbm}
\usepackage{mathrsfs}
\usepackage{amsfonts}
\usepackage{amsmath}
\usepackage{amssymb}
\usepackage{wasysym}
\usepackage{psfrag}
\usepackage{graphics,epsfig,amsmath}  
\usepackage{comment}
\usepackage{xcolor}
\usepackage{enumerate}
\newtheorem{theorem}{Theorem}

\newtheorem{Remark}{Remark}
\newtheorem{lemma}{Lemma}
\newtheorem{prop}{Proposition}

\title[Convergence of 1D mean field games]
{Convergence of one-dimensional stationary mean field games with vanishing potential}

\author[Y. Cai]{Yiru Cai}
\author[H. Qi]{Haobo Qi}
\author[Y. Tan]{Yi Tan}

\author[X. Su] {Xifeng Su}

\address{School of Mathematical Sciences\\
Beijing Normal University\\
No. 19, XinJieKouWai St.,HaiDian District\\
 Beijing 100875, P. R. China}
\email{xfsu@bnu.edu.cn}

\begin{document}

\today

\maketitle
\vspace{0.1in}

\begin{abstract}
We consider the one-dimensional stationary first-order mean-field game (MFG) system with the coupling between the Hamilton-Jacobi equation and the transport equation.  In both cases that the coupling is strictly  increasing and decreasing with respect to the density of the population, we show that when the potential vanishes the regular solution of MFG system converges to the one of the corresponding integrable MFG system.
Furthermore, we obtain the convergence rate of such limit.
\end{abstract}

{\bf Keywords: }\keywords{Mean field games, Stationary problems, regular solutions, vanishing potential}

{\bf Mathematics Subject Classification numbers:}\subjclass{
91A13,  
91A25  
}

\section{Introduction}
Lasry and Lions \cite{LL06a, LL06b, LL07}, Huang, Caines and Malham\'{e}  \cite{HMC06, HCM07} independently around the same time, first introduce notion of mean-field games  to describe on-cooperative differential games with infinitely many identical players.

In this paper, we consider the one-dimensional stationary first-order mean field games:
\begin{equation*}\label{1DMFG}
\left\{
\begin{array}{rcl}
\frac{(u_x(x)+p)^2}{2} +\epsilon V(x) &= & g(m(x)) + \bar{H},\\
\big(m(x) (u_x(x) + p) \big)_x&= & 0.
\end{array}
\right. \leqno{\hbox{($E$)}}_\epsilon
\end{equation*}
Here, the parameters $p\in \mathbb{R}, \epsilon>0$ are given, and $g: \mathbb{R}^+\rightarrow \mathbb{R}$ and $V: \mathbb{T} \rightarrow \mathbb{R}$ are given $C^\infty$ functions. The unknowns are the functions $u, m : \mathbb{T} \rightarrow \mathbb{R}$ and a real number $\bar{H}$.

We examine this standard example in MFGs in order to understand its asymptotic features when the potential vanishes and give some hints on how to deal with it in higher dimension.

We will use the current formulation as in \cite{Gomes17} and rewrite $(E)_\epsilon$ into
\begin{equation*}\label{1DMFGmodified}
\left\{
\begin{array}{rcl}
\frac{(u_x(x)+p)^2}{2} +\epsilon V(x) &= & g(m(x)) + \bar{H},\\
m(x) (u_x(x) + p) &= & j.
\end{array}
\right. \leqno{\hbox{($E$)}}_{\epsilon,j}
\end{equation*}
In general, the agent's preference depends on the current $j$, the monotonicity of $g$, and the potential $\epsilon V$. The existence result for several special forms of $g$ is discussed clearly in \cite{Gomes17}. 

It is easy to see that for every $j\in \mathbb{R}$, there exists $p=j$ such that $(E)_{0, j}$ admits a unique smooth solution $(u_0, m_0, \bar{H}_0) = (0, 1, \frac{j^2}{2} - g(1))$ when we choose the normalization $u(0) = 0$.

In the case when $g$ is increasing, for every $j\in \mathbb{R}$ there exists $p$ such that $(E)_{\epsilon, j}$ admits a unique smooth solution.

In the case when $g$ is decreasing, the existence of smooth solutions are not always possible.
Therefore, we introduce a weak notion of solution, referred as a regular solution, that is a triplet $(u, m, \bar{H})$ solves $(E)_{\epsilon, j}$ in the following sense:
\begin{itemize}
\item [(i)] $u$ is a Lipschitz viscosity solution of the first equation in $(E)_{\epsilon, j}$;

\item [(ii)] $\lim\limits_{x\rightarrow x_0^-} Du (x) \geq \lim\limits_{x\rightarrow x_0^+} Du (x)$ for any discontinuous point $x_0$ of $Du$;

\item [(iii)] $m$ is a probability density; that is 
\[
m \geq 0, \quad \int_{\mathbb{T}} m(x)\ dx =1;
\]

\item [(iv)] $m$ is a distributional solution of the second equation in $(E)_{\epsilon, j}$.
\end{itemize}
Note that this definition only works for the one-dimensional case just for the spirit of simplification and
the solutions of $(E)_{\epsilon, j}$ considered here will be unique.
Another reason why we use this notion is that we believe there may be some dynamical interpretations beyond this PDE formulation such as celebrated KAM theory (see \cite{Rafaelbook}), weak KAM theory by \cite{Fathi97a, Fathi97b, E99}. One can also refer \cite{Gomes18} for recent progress on the selection problem for stationary mean-field games. One may refer other definitions for weak solutions of MFG, see for instance \cite[Section 6]{Gomes17} and \cite{Gomes16}.

Our goal is to study the dependence of the triplet $(u_\epsilon, m_\epsilon, \bar{H}_\epsilon)$ on $\epsilon$ and to show that the viscosity solution $(u_\epsilon, m_\epsilon, \bar{H}_\epsilon)$ of $(E)_{\epsilon, j}$ converges to the solution of $(E)_{0,j}$ as $\epsilon$ vanishes and to obtain the associated convergence rate.

\begin{theorem}\label{maintheorem1}
For any given strictly increasing smooth function $g$, the mean field system $(E)_{\epsilon,j}$ admits a unique solution $(u_\epsilon, m_\epsilon, \bar{H}_\epsilon)$ and  we have
\[
\lim_{\epsilon \rightarrow 0^+}⁡{u_{\epsilon}(x)} = 0,\quad
\lim_{\epsilon \rightarrow 0^+}⁡{m_{\epsilon}(x)} = 1, \quad  \lim_{\epsilon \rightarrow 0^+}\bar{H}_{\epsilon}  = \frac{j^2}{2}-g(1).
\]
Moreover, there exists $C>0$ such that when $\epsilon>0$ small enough, we have 
\[
\left|\bar{H}_\epsilon  - \big(\frac{j^2}{2}-g(1)\big) \right| \leq C\epsilon, \quad   |m_\epsilon(x) - 1 | \leq C \epsilon, \quad |u_\epsilon(x) - 0| \leq C \epsilon.
\]
\end{theorem}

To well state the case when $g$ is strictly decreasing, we will introduce a smooth auxiliary function:
\begin{equation}\label{auxiliary function}
                  h(m):= \frac{j^2}{2m^2} - g(m),\qquad m>0.
\end{equation}
Note that when $g$ is strictly increasing, we have
$ h'(m) = -g'(m)-\frac{j^2}{m^3}<0$, but when $g$ is strictly decreasing, the monotonicity of $h$ is not available
and we need other additional assumptions to overcome this difficulty.

\begin{theorem}\label{maintheorem2}
Assume that $g$ is a strictly decreasing smooth function, $x=0$ is the single maximum of $V$ and that
\begin{itemize}
\item [(i)] $h$ satisfies the following properties when $j\neq 0$:
\begin{itemize}
\item [(a)] $h(m)$ is a strictly convex function;
\item [(b)] $\lim\limits_{m\rightarrow+\infty} h(m)\rightarrow+\infty$;
\item [(c)] $\lim\limits_{m\rightarrow0^{+}} h(m)\rightarrow+\infty$.
\end{itemize} 

\item [(ii)] $h = - g$ is convex when $j=0$.
\end{itemize}
Then, the mean field system $(E)_{\epsilon,j}$ admits a unique solution $(u_\epsilon, m_\epsilon, \bar{H}_\epsilon)$ when $\epsilon$ is sufficiently small, and we have
\[
\lim_{\epsilon \rightarrow 0^+}⁡{u_{\epsilon}(x)} = 0,\quad
\lim_{\epsilon \rightarrow 0^+}⁡{m_{\epsilon}(x)} = 1, \quad  \lim_{\epsilon \rightarrow 0^+}\bar{H}_{\epsilon}  = \frac{j^2}{2}-g(1).
\]
Moreover, there exists $C>0$ such that
\begin{itemize}
\item [(a)] when $j\neq 0$
\begin{itemize}
\item [(a1)] and $m^*(j) \neq 0$, we have
\[
\left|\bar{H}_\epsilon  - \big(\frac{j^2}{2}-g(1)\big) \right| \leq C\epsilon, \quad   |m_\epsilon(x) - 1 | \leq C \epsilon, \quad |u_\epsilon(x) - 0| \leq C \epsilon.
\]
\item [(a2)] and $m^*(j) =1$, we have
\[
\left|\bar{H}_\epsilon  - \big(\frac{j^2}{2}-g(1)\big) \right| \leq C\epsilon, \quad   |m_\epsilon(x) - 1 | \leq C \sqrt{\epsilon}, \quad |u_\epsilon(x) - 0| \leq C  \sqrt{\epsilon}.
\]
\end{itemize}

\item [(b)] when $j=0$, we have
\[
\left|\bar{H}_\epsilon  - \big(\frac{j^2}{2}-g(1)\big) \right| \leq C\epsilon, \quad   |m_\epsilon(x) - 1 | \leq C \epsilon, \quad u_\epsilon(x) = 0.
\]
\end{itemize}
\end{theorem}
Note that one can find existence results when $g(m) = -m$ in \cite{Gomes17}, but we will deal with slightly more general $g$ in this paper.

The generalizations of these results in higher dimension, such as KAM theory and weak KAM theory, will be focused in the future works. 

\begin{Remark}
The assumption that $V$ has only one maximum point can be relaxed to have finitely many maximum points.
The argument of Theorem~\ref{maintheorem2} works although we don't have uniqueness of the regular solution anymore.
\end{Remark}

This paper is organized as follows. According to the monotonicity  of coupling $g$ between the Hamilton-Jacobi equation and the transport equation, our main results (the first half of both Theorems~\ref{maintheorem1} and \ref{maintheorem2})  for the convergence of  regular solution $(u_\epsilon, m_\epsilon, \bar{H}_\epsilon)$ of $(E)_{\epsilon, j}$ as $\epsilon$ vanishes will be divided into two case in Section~\ref{increasing} and Section~\ref{decreasing} since the approaches to both cases are quite different.
In Section~\ref{convergence rate}, we obtain the explicit estimates for the convergence rate to complete the proof of Theorems~\ref{maintheorem1} and \ref{maintheorem2}.

\section{Increasing mean field games}\label{increasing}
In this section, we will consider the case when $g$  is strictly increasing with respect to the the density of the population. Heuristically, this case describes the phenomenon that agents prefer sparsely populated areas. 

The proof of the first half of Theorem~\ref{maintheorem1}  will  be divided into both cases when $j\neq 0$ and $j=0$.

\subsection{$j\neq 0$}
We will use the fundamental argument to prove the convergence results for one-dimensional mean field games. One may study the linearized operator for the higher-dimensional stability problem when $g$  is strictly increasing, which will be dealt with elsewhere.
\begin{prop}\label{increasingpositive}
For any given strictly increasing smooth function $g$ and $j\neq 0$, the mean field system $(E)_{\epsilon,j}$ admits a unique solution $(u_\epsilon, m_\epsilon, \bar{H}_\epsilon)$. Moreover, we have
\[
\lim_{\epsilon \rightarrow 0^+}⁡{u_{\epsilon}(x)} = 0,\quad
\lim_{\epsilon \rightarrow 0^+}⁡{m_{\epsilon}(x)} = 1, \quad  \lim_{\epsilon \rightarrow 0^+}\bar{H}_{\epsilon}  = \frac{j^2}{2}-g(1).
\]
\end{prop}

\begin{proof}
Without loss of generality, we consider $j > 0$. For any given $\epsilon>0$, it is shown in \cite{Gomes17} that there exists unique $p$ such that $(E)_{\epsilon,j}$ admits a unique smooth solution $(u,m,\bar{H})$. To finish the proof, it suffices to show that the triplet $(u, m, \bar{H})$ is continuous with respect to $\epsilon$. 

Construct a smooth function as follows:
\begin{equation}\label{auxiliary function for implicit function theorem}
F(x,m,\epsilon,\bar{H}) :=h(m) + \epsilon V(x) - \bar{H}.
\end{equation}
Due to fact that $\frac{\partial F}{\partial m}(x,m,\epsilon,\bar{H}) = h'(m) <0 $ and by $(E)_{\epsilon,j}$ and the implicit function theorem, there exists a continuous differentiable function $f$ such that
\[
m= f(x,\epsilon,\bar{H}),\quad \frac{\partial f}{\partial \bar{H}}(x,\epsilon,\bar{H}) =  \frac{1}{h'(m)}.
\]
Let
\[
G(x,\epsilon,\bar{H}) = \int_{\mathbb{T}}f(x,\epsilon,\bar{H})-1.
\]
We can have that $G$ is continuous on $(x, \bar{H}) \in\mathbb{T}\times \mathbb{R}$ and $G$ has continuous partial derivatives with respect to $x, \epsilon, \bar{H}$.

Notice that
\[
\frac{\partial{G}}{\partial{\bar{H}}}=\int_{\mathbb{T}}\frac{\partial{f}}{\partial{\bar{H}}}(x,\epsilon, \bar{H})  \ dx =\int_{\mathbb{T}}\frac{1}{h'(m)}  dx< 0.
\]
By the implicit function theorem, there exists a continuous function $\varphi$ such that
\[
 \bar{H}=\varphi(\epsilon).
\]
Thus, thanks to the continuity of $f$ and $\varphi$, we have
\begin{eqnarray*}
                   m= f(x,\epsilon,\bar{H}) = f(x,\epsilon,\varphi(\epsilon))
\end{eqnarray*}
which implies that $m$ is continuous in $\epsilon$.\\
Consequently,  we have
\[
\lim_{\epsilon \rightarrow 0^+}⁡{m_{\epsilon}(x)} = m_{0}(x) =1 , \quad  \lim_{\epsilon \rightarrow 0^+}\bar{H}_{\epsilon} = \bar{H}_{0} =\frac{j^2}{2} -g(1).
\]
Hence, due to the fact that $p_\epsilon=\int_{\mathbb{T}} \frac{j}{m_\epsilon(x)} dx \rightarrow j$, we obtain
\[
\lim_{\epsilon \rightarrow 0^+}⁡{u_{\epsilon}(x)} = \lim_{\epsilon \rightarrow 0^+}⁡ \int_0^x \left(\frac{j}{m_\epsilon(y)} - p \right) dy \rightarrow 0. \qedhere
\]
\end{proof}

\subsection{$j=0$}
When $j$ vanishes, we rewrite $(E)_{\epsilon,j}$ into
\begin{equation*}
\left\{
        \begin{array}{lcl}
             & \frac{( u_{x} + p)^2}{2} + \epsilon V(x) & =  g(m_{}(x)) + \bar{H}_{}, \\
             & \int_\mathbb{T} m(x)  &=   1,  \quad m \geq 0,\\
             & {m_{}}( u_{x} + p)  &=   0.
        \end{array}
        \right. \leqno{\hbox{($E$)}}_{\epsilon,0}
\end{equation*}
Then we have the following conclusion.

\begin{prop}\label{increasing with vanishing j}
For any given strictly increasing $g$ and $j = 0$, the mean field system $(E)_{\epsilon,0}$ admits a unique smooth solution $(u_\epsilon, m_\epsilon, \bar{H}_\epsilon)$. Moreover, we have
\[
\lim_{\epsilon \rightarrow 0^+}⁡{u_{\epsilon}(x)} = 0, \quad
\lim_{\epsilon \rightarrow 0^+}⁡{m_{\epsilon}(x)} = 1, \quad  \lim_{\epsilon \rightarrow 0^+}\bar{H}_{\epsilon} = -g(1)
\]
\end{prop}

\begin{proof}
(i). The existence of the unique smooth solution will be shown by the following two steps.

\textbf{Step 1.} We first claim that there exists a unique possible candidate solution for $(E)_{\epsilon, 0}$.
Actually, from $(E)_{\epsilon,0}$, we notice that
\begin{equation*}
g(m_{}(x))  =
\left\{\!\!\!\!\!
        \begin{array}{lcl}
         &\epsilon V(x)  -   \bar{H}_{} , & \text{ if } u_x+p=0,\\
         &\frac{( u_{x} + p)^2}{2} + \epsilon V(x)   -  \bar{H}_{} = g(0) , & \text{ otherwise}.\\
        \end{array}
        \right.
\end{equation*}
Since $g(m)$ is strictly increasing with respect to $m$, $[g^{-1}(\epsilon V(x)-\bar{H})]^{+}$ could be a candidate solution satisfying $(E)_{\epsilon,0}$ and is denoted by $m_\epsilon(x)$. Moreover,  for any $\epsilon \geq 0$, there must be a unique candidate $\bar{H}_\epsilon$ satisfying $(E)_{\epsilon,0}$ since the map $\bar{H}_{} \mapsto \int_{\mathbb{T}} [g^{-1}(\epsilon V(x)-\bar{H})]^{+} dx$ is strictly decreasing at its positive values. 

\textbf{Step 2.}  The unique candidate solution is smooth. Because $\int_\mathbb{T} m_\epsilon(x) =  1$ and $m_{\epsilon}(x)=[g^{-1}(\epsilon V(x)-\bar{H}_\epsilon)]^{+}$ is continuous in $x$,  there exists $x_0\in[0,1)$ such that $m_{\epsilon}(x_0)=g^{-1}(\epsilon V(x_0)-\bar{H}_\epsilon)=1$, that is
\[
\epsilon V(x_0)-\bar{H}_\epsilon = g(1).
\]
Since $g(m)$ is strictly increasing, we have
$
                   -g(1)<-g(0).
$
Thus, let $\epsilon_0=\frac{g(1)-g(0)}{\max\limits_\mathbb{T} V - \min\limits_\mathbb{T} V}>0$, for any $0\leq \epsilon< \epsilon_0$, we have
\begin{eqnarray*}
                   \bar{H}_\epsilon = \epsilon V(x_0) - g(1)\leq-g(0)+\epsilon \min\limits_\mathbb{T} V \Longleftrightarrow g^{-1}(\epsilon \min_{\mathbb{T}} V(x)-\bar{H}_\epsilon)\geq 0.
\end{eqnarray*}
which implies $g^{-1}(\epsilon V(x)-\bar{H}_\epsilon)\geq 0$ by the monotonicity of $g$.

Therefore, the function
\[
m_\epsilon(x)=g^{-1}(\epsilon V(x)-\bar{H}_\epsilon),\qquad0<\epsilon< \epsilon_0
\]
is smooth. Hence, the unique candidate solution $(0, m_\epsilon, \bar{H}_\epsilon)$ is a solution of $(E)_{\epsilon, 0}$.

(ii). Using the same argument as in Proposition~\ref{increasingpositive}, we obtain that 
$(u_\epsilon (x), m_\epsilon(x), \bar{H}_\epsilon)$ is continuous in $\epsilon\in[0,\epsilon_0)$.
As a result, we have
\[
\lim_{\epsilon \rightarrow 0^+}⁡{u_{\epsilon}(x)} = 0, \quad
\lim_{\epsilon\rightarrow0}m_\epsilon(x) =1, \quad \lim_{\epsilon\rightarrow0}\bar{H}_\epsilon=-g(1). \qedhere
\]
\end{proof}

\section{Decreasing mean field games}\label{decreasing}
In general, an interesting new phenomenon, called an unhappiness trap has been discovered in \cite{Gomes17}. That is, when the current $j$ is smaller, the density $m(x)$ is larger where $\epsilon V(x)$ is smaller; when the current $j$ is large, the density $m(x)$ is larger where $\epsilon V(x)$ is larger; in the  intermediate case, both situations are mixed.

However, we observe that when $\epsilon V(x)$ is small, the value of the current $j$ will not bring much trouble to us and the density of the population is close to even distribution. 

The proof of the first half of Theorem~\ref{maintheorem2} will be  divided into two cases when $j\neq 0$ and $j=0$.

\subsection{$j\neq 0$} \label{decreasing and j is not zero}
To consider the case when $g$ is decreasing, we will first impose some additional hypotheses on $g$. In fact, instead of imposing direct hypotheses on $g$, we find that it is more convenient to assume the auxiliary function $h$ satisfies the properties (a), (b) and (c) stated in Theorem~\ref{maintheorem2}.

It is easy to see that  for any $j>0$,  $h$ has  a unique minimum point denoted by $m^*$ and so one can think of $m^*$ as a function of $j$ and write $m^*= m^*(j)$.
Moreover, for any $m>0$, we have $h(m) \geq h(m^*(j))$.

\begin{lemma}\label{existence with non-vanishing j}
Assume that $x=0$ is the single maximum of $V$ and that (a), (b) and (c) hold. Then $(E)_{\epsilon, j}$ admits a unique solution.
\end{lemma}

\begin{proof}
One can easily find a lower bound of $\bar{H}$, that is
\[
\bar{H} \geq \bar{H}_\epsilon^{cr}:=\epsilon \max\limits_{\mathbb{T}} V + h(m^*).
\]since one can rewrite the first equation of $(E)_{\epsilon, j}$ as
\begin{equation}\label{transformed equation}
                h(m(x))  \equiv  \frac{j^2}{2 (m(x))^2} - g(m(x))  = \bar{H} - \epsilon V(x)  .
\end{equation}
Let $m_{\bar{H}}^-(x)$ and $m_{\bar{H}}^+(x)$ be the two solutions of \eqref{transformed equation} for any given $x\in \mathbb{T}$ and we have that $m_{\bar{H}}^- \leq m^*(j) \leq  m_{\bar{H}}^+$. Let
\[
\alpha_{\bar{H}}^{\pm} = \int_0^1 m_{\bar{H}}^{\pm}(x) \ dx.
\]
Note that if $V(x)$ is not a constant, then for any $j>0, \epsilon>0$, $\alpha_{\bar{H}}^-< \alpha_{\bar{H}}^+$.

In particular, let $m_{\bar{H}_\epsilon^{cr}}^-$ and $m_{\bar{H}_\epsilon^{cr}}^+$ be the two solutions of \eqref{transformed equation} when $\bar{H} = \bar{H}_\epsilon^{cr}$. We then define the following two fundamental  quantities to analyse the existence of the possible solutions
\begin{equation}\label{2auxiliary functions}
\alpha_\epsilon^+(j) := \int_0^1 m_{\bar{H}_\epsilon^{cr}}^+(x) \ dx, \quad \alpha_\epsilon^-(j) := \int_0^1 m_{\bar{H}_\epsilon^{cr}}^-(x) \ dx.
\end{equation}

\textbf{Step 1. } We first claim that $\lim\limits_{\epsilon \rightarrow 0^+} m_{\bar{H}_\epsilon^{cr}}^{\pm} = m^*(j)$.
In fact, since 
\[
h(m_{\bar{H}_\epsilon^{cr}}^{\pm}) = h(m^*(j)) + \epsilon (\max_{\mathbb{T}} V - V(x) ),
\]
we know that $\lim\limits_{\epsilon \rightarrow 0^+} h(m_{\bar{H}_\epsilon^{cr}}^{\pm} ) =h(m^*(j))$.
Suppose $m_{\bar{H}_\epsilon^{cr}}^{\pm}$ does not converge to $m^*(j)$ as $\epsilon$ goes to $0^+$, then one can find  an accumulate point $m^{**}$. But due to the continuity of $h$, we should have
$h(m^{**}) = h(m^*(j))$. Moreover, by the unqueness of $m^*(j)$, $m^{**} = m^*(j)$, which implies that
$\lim\limits_{\epsilon \rightarrow 0^+} m_{\bar{H}_\epsilon^{cr}}^{\pm} = m^*(j)$.

\textbf{Step 2. } We are ready to show the proof of the lemma in the following three cases according to the value of $m^*(j)$.

\begin{itemize}
\item [(i)] Suppose $m^*(j)>1$.  Due to \textbf{Step 1 }, one can choose $\epsilon>0$ small enough such that $ m^*(j) \geq \alpha_\epsilon^-(j)  >1$.
On the other hand, it is easy to see that the map $\bar{H} \mapsto \alpha_{\bar{H}}^-$ is strictly decreasing and the image is $(0, m^*(j)]$. Hence, one can find $\bar{H}_\epsilon$ such that $\alpha_{\bar{H}_\epsilon}^- =1$.
Furthermore, we obtain
\[
u_\epsilon (x) = \int_0^x \frac{j}{m_{\bar{H}_\epsilon}^-(y)} \ dy - p_j x ,\quad\text{ where  } p_j = \int_0^1\frac{j}{m_{\bar{H}_\epsilon}^-(y)} \ dy.
\]

\item [(ii)] Suppose $m^*(j)<1$.  Due to \textbf{Step 1 }, one can choose $\epsilon>0$ small enough such that $ m^*(j) \leq \alpha_\epsilon^+(j)  <1$.
On the other hand, it is easy to see that the map $\bar{H} \mapsto \alpha_{\bar{H}}^+$ is strictly increasing and the image is $[m^*(j), +\infty)$. Hence, one can find $\bar{H}_\epsilon$ such that $\alpha_{\bar{H}_\epsilon}^+ =1$.
Furthermore, we obtain
\[
u_\epsilon (x) = \int_0^x \frac{j}{m_{\bar{H}_\epsilon}^+(y)} \ dy - p_j x ,\quad\text{ where  } p_j = \int_0^1\frac{j}{m_{\bar{H}_\epsilon}^+(y)} \ dy.
\]

\item [(iii)] Suppose $m^*(j) =1$. Firstly, we claim that  $(E)_{\epsilon, j}$ admits no regular solution when $\bar{H} > \bar{H}_\epsilon^{cr}$. Suppose by contradiction that we have a regular solution $(u_\epsilon, m_\epsilon, \bar{H}_\epsilon)$ with $\bar{H}_\epsilon > \bar{H}_\epsilon^{cr}$ and $p\in \mathbb{R}$. 
Consequently, 
\begin{equation}\label{negative jump of m}
m_{\bar{H}_\epsilon}^-(x) < 1< m_{\bar{H}_\epsilon}^+(x), \quad
\inf_{x\in\mathbb{T}} ( m_{\bar{H}_\epsilon}^+(x) - m_{\bar{H}_\epsilon}^-(x) ) >0.
\end{equation}
Obviously,  the density has the following form
\begin{equation}
m_\epsilon(x) = m_{\bar{H}_\epsilon}^+(x)\  \chi_{E}(x) + m_{\bar{H}_\epsilon}^-(x)\  \chi_{\mathbb{T}\setminus E}(x), 
\end{equation} where $E$ is some subset of $\mathbb{T}$.
Moreover, we obtain
\[
\begin{split}
&1 = \int_{\mathbb{T}} m_\epsilon(x) \ dx = \int_{E} m_{\bar{H}_\epsilon}^+(x)\  dx + \int_{\mathbb{T}\setminus E} m_{\bar{H}_\epsilon}^-(x)\  dx, \\
& \int_{\mathbb{T}} m_{\bar{H}_\epsilon}^-(x)\  dx <   1
<  \int_{\mathbb{T}} m_{\bar{H}_\epsilon}^+(x)\  dx.
\end{split}
\]
Therefore, the Lebesgue measure of $E$ is in $(0,1)$. Hence, one can find $e\in \mathbb{T}$ such that
for any $\eta>0$, we have 
\[
(e-\eta, e) \cap E \neq \emptyset, \quad (e, e+\eta)  \cap (\mathbb{T} \setminus E) \neq \emptyset.
\]
Combining with \eqref{negative jump of m}, we get that $m_\epsilon(e^-) - m_\epsilon(e^+) <0$. So 
\[
(u_\epsilon)_x(e^-) - (u_\epsilon)_x(e^+)  = j \ \left( \frac{1}{m_\epsilon(e^-)} -  \frac{1}{m_\epsilon(e^+)}\right) <0,
\]which contradicts the regularity assumption of $u_\epsilon$.

Hence, in other words, if we want to find a solution of $(E)_{\epsilon, j}$, it is necessary to have $\bar{H}_\epsilon = \bar{H}_\epsilon^{cr} $.

Notice that $m_\epsilon(x)$ can switch from $m_{\bar{H}_\epsilon^{cr}}^+(x)$ to $m_{\bar{H}_\epsilon^{cr}}^-(x)$ if and only if the switch point is a continuity point, which implies $m_{\bar{H}_\epsilon^{cr}}^+(x)=m_{\bar{H}_\epsilon^{cr}}^-(x)$. This case can only happen when $V$ meets its maximum.
Since $x=0$ is the single maximum of $V$, one can have as the only possible candidate solution the piecewise function of the form below:
\begin{equation*}
m_\epsilon(x)=
\left\{\!\!\!\!\!
        \begin{array}{lcl}
         &m_{\bar{H}_\epsilon^{cr}}^-(x), & \text{ for } x\in [0,d_\epsilon),\\
         &m_{\bar{H}_\epsilon^{cr}}^+(x) , & \text{ for } x\in [d_\epsilon, 1) .\\
        \end{array}
        \right.
\end{equation*}
Our next step is to find an appropriate $d_\epsilon\in (0,1)$. In fact, let us define
\[
\varphi(d) := \int_0^1 m_\epsilon(x) dx = \int_0^d  m_{\bar{H}_\epsilon^{cr}}^-(x) dx + \int_d^1 m_{\bar{H}_\epsilon^{cr}}^+(x)  dx
\]which is differentiable in $(0,1)$. Due to $\varphi(0) >1, \varphi(1) <1$ and $\varphi'(d) = m_{\bar{H}_\epsilon^{cr}}^-(x)- m_{\bar{H}_\epsilon^{cr}}^+(x) <0$, there exists a unique $d_\epsilon\in(0,1)$ such that $\varphi(d_\epsilon) = 1$. Furthermore, we have
\[
u_\epsilon(x) = \int_0^x \frac{j}{m_\epsilon(y)} dy- p_j x \quad \text{ where } p_j = \int_0^1 \frac{j}{m_\epsilon(y)} dy. \qedhere
\]
\end{itemize}
\end{proof}

Moreover, we obtain:
\begin{lemma}\label{convergence for the decreasing g and nonvanishing j}
Fix $j\neq 0$. Suppose that $x=0$ is the single maximum of $V$. 
Assume that (a), (b) and (c) hold. Then 
\[
\lim_{\epsilon\rightarrow0}u_\epsilon(x) = 0, \quad
\lim_{\epsilon\rightarrow0}m_\epsilon(x) = 1, \quad 
\lim\limits_{\epsilon\rightarrow0}\bar{H}_\epsilon =\frac{j^2}{2}-g(1).
\]
\end{lemma}

\begin{proof}
We now show the continuity of the obtained solution $(u_\epsilon(x), m_\epsilon(x), \bar{H}_\epsilon)$ of $(E)_{\epsilon, j}$ with respect to $\epsilon$.
We will divide the proof into three cases according to the value of $m^*(j)$ as in the proof of Lemma~\ref{existence with non-vanishing j}.

(i). For the first two cases, we consider the function $F$ defined in \eqref{auxiliary function for implicit function theorem}. It is easy to see that
\begin{eqnarray*}
                   \frac{\partial F}{\partial m}(x, m, \epsilon, \bar{H}) = h'(m) = -\frac{j^2}{m^3}-g'(m) = 
                  \left\{
        \begin{array}{lcl}
             >0  \qquad 0<m<m^*\\
             <0   \qquad m>m^*
        \end{array}
        \right.
\end{eqnarray*} because $m^*$ is a minimum point of $h$.
Applying the same argument as in the proof of Proposition~\ref{increasingpositive} and the fact that $m^*(j)>1$ and $m^*(j)<1$  implies $m_\epsilon(x) = m_{\bar{H}_\epsilon}^-(x) < m^*(j)$ and $m_\epsilon(x) = m_{\bar{H}_\epsilon}^+(x) > m^*(j)$ respectively, one can obtain the continuity of $(u_\epsilon(x), m_\epsilon(x), \bar{H}_\epsilon)$ in $\epsilon$.

(ii). When $m^*(j)=1$, we notice that 
\[
m_\epsilon(x) = m_{\bar{H}_\epsilon^{cr}}^+(x)\  \chi_{[d_\epsilon, 1)}(x) + m_{\bar{H}_\epsilon^{cr}}^-(x)\  \chi_{[0, d_\epsilon)}(x), 
\]where $d_\epsilon\in (0,1)$ is a uniquely determined number.
Using the fact that $\lim\limits_{\epsilon \rightarrow 0^+} m_{\bar{H}_\epsilon^{cr}}^{\pm} = m^*(j)$ in the first step of the proof of Lemma~\ref{existence with non-vanishing j}, we have 
\[
\lim_{\epsilon\rightarrow0^+}m_\epsilon(x) = 1,
\]
which finishes the proof of the lemma.
\end{proof}

\subsection{$j=0$}
Besides that $g$ is strictly decreasing, we will assume furthermore that $h=-g$ is convex, which is consistent with the hypotheses in the case when $j\neq 0$ in the last section above.
 
Now, we consider the system $(E)_{\epsilon, 0}$ and obtain the lower bound of $\bar{H}$ there.
In fact, due to the inequality
\begin{equation}\label{inequality}
              -g( m(x) )=\bar{H}-\epsilon V(x)-\frac{(u_{x} + p)^2}{2} \leq \bar{H} - \epsilon V(x),
\end{equation}
we have
\begin{equation*}
            \bar{H} \geq \epsilon \max\limits_{\mathbb{T}}V(x)-g(0).
\end{equation*}
On the other hand, we integrate both sides of \eqref{inequality} over $\mathbb{T}$ and get
\begin{equation*}
            \bar{H}  \geq \epsilon\int\limits_{\mathbb{T}}V(x)\  dx - \int\limits_{\mathbb{T}}g(m(x))\  dx.
\end{equation*}
Thus, we obtain the relation between $\bar{H}$ and $m$
\begin{equation*}
         \bar{H}  \geq \max\left\{\epsilon\max_{\mathbb{T}} V(x)-g(0), -\int_{\mathbb{T}}g(m(x))\  dx +\epsilon\int_{\mathbb{T}}V(x) \ dx \right\} =: \bar{H}^{0}_m.
\end{equation*}
We assert that, when $\epsilon$ is small enough, $(E)_{\epsilon, 0}$ has a viscosity solution if and only if $\bar{H}=\bar{H}^0_m$.

\begin{lemma}\label{decreasing with vanishing j}
\begin{itemize}
\item [(i)] If $\bar{H} >\bar{H}^0_m$, $(E)_{\epsilon, 0}$ does not have any viscosity solution.
\item [(ii)] If $\bar{H} =\bar{H}^0_m$, there exists $\epsilon_0>0$ such that when $0<\epsilon < \epsilon_0$, $(E)_{\epsilon, 0}$ has  a unique viscosity solution.
\end{itemize}
\end{lemma}

\begin{proof}
(i). We follow the argument in the proof of Proposition 5.3 in \cite{Gomes17}. We suppose by contradiction that $\bar{H} >\bar{H}^0_m$ and $(E)_{\epsilon, 0}$ has a viscosity solution $(u_\epsilon, m_\epsilon, \bar{H}_\epsilon)$. 

We first give a description of the set of points where the density $m_\epsilon$ vanishes.
In fact, we define a set 
\[
Z:= \{x\in\mathbb{T}~|~(u_\epsilon)_x(x)+p\neq0\} \subset \{x\in\mathbb{T}~|~m_\epsilon(x) = 0\}
\]where the inclusion is due to the second equation of $(E)_{\epsilon, 0}$.

By assumption, we have
\begin{equation*}
\begin{split}
         \bar{H}_\epsilon &= \int_{\mathbb{T}}\frac{( (u_\epsilon)_x(x)+p)^2}{2} \ dx - \int_{\mathbb{T}}g(m_\epsilon(x)) \ dx + \epsilon \int_{\mathbb{T}}V(x)\  dx \\
&>- \int_{\mathbb{T}}g(m_\epsilon (x) ) \ dx+\epsilon \int_{\mathbb{T}} V(x)\ dx,
\end{split}
\end{equation*}
which implies that $Z$ has positive Lebesgue measure. 

On the other hand, taking any $x\in \mathbb{T}$ such that $m_\epsilon(x) =0$, we have 
\[
\begin{split}
\frac{( (u_\epsilon)_x(x)+p)^2}{2} &= \bar{H}_\epsilon + g(0)  - \epsilon V(x)\\
&> \epsilon [ \max_{\mathbb{T}} V(x)  - V(x) ] >0,
\end{split}
\]
which shows that $Z= \{x\in\mathbb{T}~|~m_\epsilon(x) = 0\}$ with Lebesgue measure in $(0,1)$. 

Secondly, notice that $(u_\epsilon)_x + p$ takes either the value $\sqrt{2[\bar{H}_\epsilon-\epsilon V(x)+g(0)]} $ or $-\sqrt{2[\bar{H}_\epsilon-\epsilon V(x)+g(0)]} $ on $Z$. In the follwing, we want to show that these two cases are impossible. 

Suppose there is some point such that the latter case holds. Without loss of generality, we define
\[
         e:=\sup \left\{ x\in Z\cap (0,1) ~\big|~(u_\epsilon)_x(x) +p=-\sqrt{2[\bar{H}_\epsilon-\epsilon V(x)+g(0)]} \right\}
\]
then, at $x= e$, the jump of $(u_\epsilon)_x+p$ is of size $\sqrt{2(\bar{H}_\epsilon-\epsilon V(x)+g(0))}$ or $2\sqrt{2(\bar{H}_\epsilon-\epsilon V(x)+g(0))}$ at $x=e$, which is a contradiction to the definition of semi-concavity.

Suppose that $(u_\epsilon)_x+p$ can only take the value of $\sqrt{2(\bar{H}_\epsilon-\epsilon V(x)+g(0))}$ or $0$ on $\mathbb{T}$. Hence, there must be a point $x\in\mathbb{T}$ such that $(u_\epsilon)_x+p$ changes from $0$ to $\sqrt{2(\bar{H}_\epsilon-\epsilon V(x)+g(0))}$, which is also a contradiction. Consequently, when $\bar{H} >\bar{H}^0_m$, $(E)_{\epsilon, 0}$ does not have a semi-concave solution.

(ii). Since $g$ is concave and $\int_{\mathbb{T}} m(x) dx =1$, by Jensen's inequality, we obtain
\[
\int_{\mathbb{T}}g(m(x)) \ dx \leq g\left(\int_{\mathbb{T}} m(x) dx\right) = g(1).
\]
Consequently, since $g$ is strictly decreasing, there exists $\epsilon_0>0$ such that for every $0\leq\epsilon<\epsilon_0$, we have
\[
         \epsilon [\max_{\mathbb{T}} V(x) - V(x)] < g(0)  -g(1) .
\]
Hence, integrating over $\mathbb{T}$ for the above inequality and combining Jensen's inequality,  we get
\[
\begin{split}
 \epsilon \max_{\mathbb{T}} V(x) -  \epsilon \int_{\mathbb{T}} V(x) dx  < g(0)  -g(1) \leq  g(0) - \int_{\mathbb{T}}g(m(x)) dx,
\end{split}
\]which implies
\[
\bar{H}^{0}_m = -\int_{\mathbb{T}}g(m(x))\  dx +\epsilon\int_{\mathbb{T}}V(x) \ dx .
\]

Now we come to show the existence of  solution to $(E)_{\epsilon, 0}$ satisfying $\bar{H} =\bar{H}^{0}_m$.
From $(E)_{\epsilon, 0}$, we have
\[
\begin{split}
 \bar{H} &= \int_{\mathbb{T}}\frac{( u_{x} + p)^2}{2} dx + \epsilon \int_{\mathbb{T}} V(x) dx -  \int_{\mathbb{T}}g(m_{}(x)) dx \\
&= -\int_{\mathbb{T}}g(m(x))\  dx +\epsilon\int_{\mathbb{T}}V(x) \ dx.
\end{split}
\]
Thus, $u_x+p=0$ holds almost everywehere, and then we obatin that $p=0, u(x)=0$.
Hence, since $g$ is strictly decreasing, 
\[
         \epsilon V(x) = g(m(x))+\bar{H} \Longrightarrow m(x) = g^{-1}(\epsilon V(x)-\bar{H}) .
\]
Therefore, in order to find the solution of $(E)_{\epsilon, 0}$, we need to find $\bar{H}$ such that $\int_{\mathbb{T}}g^{-1}(\epsilon V(x) - \bar{H})dx =1$ holds.
In fact, we note that  the map $\bar{H} \mapsto \int_{\mathbb{T}} g^{-1}(\epsilon V(x)-\bar{H}) dx$ is strictly increasing, so there exists a unique $\bar{H}_*$ such that 
$\int_{\mathbb{T}}g^{-1}(\epsilon V(x) - \bar{H})dx =1$ holds.

Therefore, $(E)_{\epsilon, 0}$ has  a unique solution 
\[
(u_\epsilon(x),m_\epsilon(x) , \bar{H}_\epsilon) = (0, g^{-1}(\epsilon V(x)-\bar{H}_*) , \bar{H}_*). \qedhere
\] 
\end{proof}

As a conclusion, we obtain
\begin{lemma}
Assume that $j = 0$ and $-g$ is convex. Then 
\[
\lim_{\epsilon\rightarrow0}u_\epsilon(x) = 0, \quad
\lim_{\epsilon\rightarrow0}m_\epsilon(x) = 1, \quad \lim\limits_{\epsilon\rightarrow0}\bar{H}_\epsilon = -g(1).
\]
\end{lemma}

\section{The convergence rate with vanishing potential}\label{convergence rate}
In this section, we will obtain the convergence rate with respect to the parameter $\epsilon$ according to whether $j$ vanishes or not.
The reason is that, based on the analysis of the above sections, we note that when $j=0$, the proof of convergence rate is the same.
When $j\neq 0$, the proof of convergence rate  is divided into the cases when $g$ is strictly  increasing and decreasing.

Firstly, we have
\begin{lemma}
When $j=0$, for any strictly increasing function $g$ (or strictly decreasing  and concave function $g$), there exists an $\epsilon_0>0$, such that when $0<\epsilon<\epsilon_0$, the solution of  $(m_\epsilon, u_\epsilon, \bar{H}_\epsilon)$ of $(E)_{\epsilon, 0}$ has the following estimates:
\begin{align}
&|\bar{H}_\epsilon-\bar{H}_0|=|\epsilon V(x_0)|\leq \left(\max_\mathbb{T} |V| \right) \ \epsilon,\\
&|m_\epsilon-1|\leq \frac{2 (\max_\mathbb{T} V- \min_\mathbb{T} V)}{|g'(1)|} \ \epsilon, \\
& u_\epsilon=u_0=0. 
\end{align}
\end{lemma}

\begin{proof}
(i).  When $j=0$, due to Proposition~\ref{increasing with vanishing j} and Lemma~\ref{decreasing with vanishing j}, we know that there exists an $\epsilon_0>0$ such that when $0<\epsilon<\epsilon_0$, we have
\begin{equation*}
1=\int_\mathbb{T}g^{-1}(\epsilon V(x)-\bar{H}_\epsilon)dx.
\end{equation*}
Due to the continuity of $g^{-1}$, there exists $x_0\in \mathbb{T}$ such that
\[
g(1)=\epsilon V(x_0)-\bar{H}_\epsilon.
\]
Thus, we have
\[
|\bar{H}_\epsilon-\bar{H}_0|=|\bar{H}_\epsilon-(-g(1))|=|\epsilon V(x_0)|\leq \epsilon \max_\mathbb{T} |V(x)|.
\]

(ii). On the one hand, we note that 
\[
g(m_\epsilon(x))=\epsilon V(x)-\bar{H}_\epsilon=\epsilon V(x)+g(1)-\epsilon V(x_0)
\] and so 
\[
|g(m_\epsilon(x))-g(1)|=\epsilon |V(x)-V(x_0)|\leq (\max_\mathbb{T} V-\min_\mathbb{T} V)\epsilon .
\]
On the other hand, since $g$ is strictly increasing or decreasing,  for any $|\frac{g'(1)}{2}| \geq\eta>0$, there exists $\delta>0$ such that for any $0<|m_\epsilon(x) -1| <\delta$, we have
\[
\left|\frac{g(m_\epsilon(x))-g(1)}{m_\epsilon(x)-1} - g'(1) \right| < \eta.
\]
Therefore,  we obtain
\[
\left|\frac{g(m_\epsilon(x))-g(1)}{m_\epsilon(x)-1} \right| > |g'(1)| -\eta \geq \frac{|g'(1)|}{2}.
\]
Hence, 
\[
 |m_\epsilon(x)-1|\leq \frac{2 (\max_\mathbb{T} V- \min_\mathbb{T} V)}{|g'(1)|} \ \epsilon.
\]

(iii). It is obvious that  $u_\epsilon =u_0 =0$.
\end{proof}

Now we come to the cases when $j\neq 0$.
\begin{lemma}\label{convergence rate for increasing g}
When $j\neq 0$ and $\epsilon$ is small enough, for any given strictly increasing function $g$, the solution $(m_\epsilon,u_{\epsilon},\bar{H}_\epsilon)$ of $(E)_{\epsilon, j}$ has the following estimates:
\begin{align}
&|\bar{H}_\epsilon-\bar{H}_0|\leq\max_\mathbb{T} |V| \ \epsilon, \\
& |m_\epsilon(x)-1|\leq \frac{2 (\max_\mathbb{T} V- \min_\mathbb{T} V)}{|h'(1)|} \ \epsilon ,\\
&|u_{\epsilon}(x) -0|\leq  8j
\frac{\max_\mathbb{T}V - \min_\mathbb{T}V}{|h'(1)|}\epsilon.
\end{align}
\end{lemma}

\begin{proof}
(i). We first note that $m_\epsilon(x) = h^{-1} (\bar{H}_\epsilon-\epsilon V(x))$ since $g$ is strictly increasing.
Due to the fact that $\int_\mathbb{T}m_\epsilon(x) dx =1$ and that $h^{-1}$ is continuous, there exists $x_0\in\mathbb{T}$, such that $m_\epsilon(x_0) = h^{-1}(\bar{H}_\epsilon - \epsilon V(x_0)) = 1$, that is 
\[
\bar{H}_\epsilon-\epsilon V(x_0)=h(1)=\frac{j^2}{2}-g(1)=\bar{H}_0.
\]
Hence, we obtain
\[
|\bar{H}_\epsilon-\bar{H}_0| = |\epsilon V(x_0)| \leq \max_\mathbb{T}|V|\ \epsilon.
\]

(ii). Moreover, from (i), we  have 
\[
|h(m_\epsilon(x))-h(m_\epsilon(x_0))|=\epsilon |V(x)-V(x_0) | \leq(\max_TV-\min_TV) \ \epsilon.
\]
Since $h$ is strictly decreasing and $m_\epsilon(x)$ converges uniformly to 1 in $x$, for any $|\frac{h'(1)}{2}| \geq\eta>0$, there exists $\delta>0$ such that for any $0<|m_\epsilon(x) -1| <\delta$, we have
\[
\left|\frac{h(m_\epsilon(x))-h(1)}{m_\epsilon(x)-1} - h'(1) \right| < \eta.
\]
Thus,
\[
 |m_\epsilon(x)-1|\leq \frac{2 (\max_\mathbb{T} V- \min_\mathbb{T} V)}{|h'(1)|} \ \epsilon.
\]

(iii). Due to Proposition~\ref{increasingpositive}, we know
\[
u_{\epsilon}(x)= \int_{0}^x \frac{j}{m_\epsilon(y)} dy - p_\epsilon x ,
\]
where $p_\epsilon=\int_T\frac{j}{m_\epsilon(x)} dx$.
Since $m_\epsilon$ is continuous, $p_\epsilon=\frac{j}{m_\epsilon (\beta_\epsilon)}$ for some $\beta_\epsilon\in \mathbb{T}$.

Note that $m_\epsilon(x)>0$ is continuous and we get $\min_\mathbb{T} m_\epsilon >0$. Therefore,
\[
\begin{split}
|u_{\epsilon}(x) - 0| &= \left| \int_0^x \bigg[\frac{j}{m_\epsilon(y)}  - \frac{j}{m_\epsilon(\beta_\epsilon)} \bigg] dy \right| 
 \leq j \int_0^x \left|\frac{m_\epsilon(\beta_\epsilon)-m_\epsilon(y) }{m_\epsilon(y) m_\epsilon(\beta_\epsilon)} \right| dy \\
&\leq j \int_0^x \left( \left|\frac{m_\epsilon(\beta_\epsilon)-1}{m_\epsilon(y) m_\epsilon(\beta_\epsilon)} \right|+\left|\frac{m_\epsilon(y)-1}{m_\epsilon(y)  m_\epsilon (\beta_\epsilon)}\right| \right) dy \\
&\leq \frac{2j}{(\min_{\mathbb{T}} m_\epsilon)^2 }
\frac{\max_\mathbb{T}V - \min_\mathbb{T}V}{|h'(1)|}\epsilon\\
&\leq 8j \frac{\max_\mathbb{T}V - \min_\mathbb{T}V}{|h'(1)|}\epsilon.
\end{split}
\]
The last inequality holds when $\epsilon$ is small enough since $m_\epsilon(x)$ converges uniformly to 1 in $x$.
\end{proof}

To end this section, we will deal with the case when $g$ is strictly decreasing and $j\neq 0$. The idea of the proof is quite different from the above cases when $m^*(j) =1$.
\begin{lemma}
Assume that $x=0$ is the single maximum of $V$, $g$ is strictly decreasing and that (a), (b) and (c) hold. When $j\neq 0$, the solution $(m_\epsilon,u_{\epsilon},\bar{H}_\epsilon)$ of $(E)_{\epsilon, j}$ has the following estimates:
\begin{itemize}
\item [\textbf{Case 1.}] When $m^*(j)\neq1$, we have
\begin{align}
&|\bar{H}_\epsilon-\bar{H}_0|\leq\max_\mathbb{T} |V| \ \epsilon, \\
& |m_\epsilon(x)-1|\leq \frac{2 (\max_\mathbb{T} V- \min_\mathbb{T} V)}{|h'(1)|} \ \epsilon ,\\
&|u_{\epsilon}(x) -0|\leq  8j 
\frac{\max_\mathbb{T}V - \min_\mathbb{T}V}{|h'(1)|}\epsilon.
\end{align}

\item [\textbf{Case 2.}]  When $m^*(j)=1$, we have
\begin{align}
&|\bar{H}_\epsilon-\bar{H}_0|\leq\max_\mathbb{T} |V| \ \epsilon, \\
& |m_\epsilon(x)-1|\leq \frac{2 \sqrt{\max_\mathbb{T} V - \min_\mathbb{T} V}}{\sqrt{h''(1)}} \sqrt{\epsilon} ,\\
&|u_{\epsilon}(x) -0|\leq 16 j \frac{ \sqrt{\max_\mathbb{T} V - \min_\mathbb{T} V}}{\sqrt{h''(1)}} \sqrt{\epsilon}.
\end{align}
\end{itemize}
\end{lemma}

\begin{proof}
(i). When $m^*(j)> 1$ ( or $m^*(j) < 1$ ),  due to Lemma~\ref{existence with non-vanishing j}, we know  that $m_{\epsilon}(x)= m^{+}_{\bar{H}_\epsilon}(x) \geq m^*(j)$ (or $m_{\epsilon}(x)= m^{-}_{\bar{H}_\epsilon}(x) \leq m^*(j)$) is the continuous solution to $(E)_{\epsilon, j}$. Due to \eqref{transformed equation} and $h$ is strictly increasing on $\{ m^{+}_{\bar{H}_\epsilon}(x) : x\in \mathbb{T} \}$ (or strictly decreasing on $\{ m^{-}_{\bar{H}_\epsilon}(x) : x\in \mathbb{T} \}$), Case 1 follows directly from the argument in the proof of Lemma~\ref{convergence rate for increasing g}.

(ii). When $m^*(j)=1$, we have
\[
\bar{H}_\epsilon = \bar{H}_\epsilon^{cr} =\epsilon \max\limits_{\mathbb{T}} V + h(1),
\]
which implies
\[
\left|\bar{H}_\epsilon-\bar{H}_0 \right| \leq \max_\mathbb{T} |V| \ \epsilon.
\]

We now show the convergence rate of  $m_\epsilon \rightarrow 1$. Note that
\[
\lim_{m\rightarrow 1}\frac{h(m)-h(1)}{(m-1)^2}=\frac{1}{2}h''(1) > 0.
\]
Therefore, for any $\frac{1}{4}h''(1) \geq \eta>0$, there is a $\delta>0$ such that,  when $0<|m_\epsilon(x) -1|<\delta$, we have
\[
\left| \frac{h(m_\epsilon(x))-h(1)}{(m_\epsilon(x)-1)^2} - \frac{1}{2}h''(1) \right|  < \eta,
\]
which implies
\[
|m_\epsilon(x) - 1| \leq \frac{2 \sqrt{ |h(m_\epsilon(x))-h(1)| }}{\sqrt{h''(1)}} \leq \frac{2 \sqrt{\max_\mathbb{T} V - \min_\mathbb{T} V}}{\sqrt{h''(1)}} \sqrt{\epsilon} .
\]

The last part is to show the convergence rate of $u_\epsilon \rightarrow 0$. Notice first that, 
for any $\epsilon>0, x\in\mathbb{T}$, we have
\[
\left((u_\epsilon)_x(x)+p_\epsilon \right) \ m_\epsilon(x) = j, \qquad \text{ where }p_\epsilon=\int_{\mathbb{T}}\frac{j}{m_\epsilon(x)}dx.
\]
Consequently,  by Lemma~\ref{convergence for the decreasing g and nonvanishing j}, when $\epsilon$ is small enough, we get $|m_\epsilon(x) | > \frac{1}{2}$. Therefore,
\[
\begin{split}
\left|(u_\epsilon)_x(x)-0 \right| = & \left|\frac{j}{m_\epsilon(x)}-p_\epsilon \right| 
=j \left|\int_\mathbb{T} \left(\frac{1}{m_\epsilon(x)}-\frac{1}{m_\epsilon(y)} \right) dy\right| \\
\leq & j \int_\mathbb{T} \left|\frac{1}{m_\epsilon(x)}-\frac{1}{m_\epsilon(y)}\right| dy
=   j \int_\mathbb{T}    \frac{|m_\epsilon(y) -1| + | 1- m_\epsilon(x)|}{ |m_\epsilon(x) \ m_\epsilon(y)| } dy\\
\leq &16 j \frac{ \sqrt{\max_\mathbb{T} V - \min_\mathbb{T} V}}{\sqrt{h''(1)}} \sqrt{\epsilon}.
\end{split}
\]
Hence, 
\[
| u_\epsilon(x)-0|=\left|\int_0^{x}(u_\epsilon)_y(y)dy \right| \leq 16 j \frac{ \sqrt{\max_\mathbb{T} V - \min_\mathbb{T} V}}{\sqrt{h''(1)}} \sqrt{\epsilon}. \qedhere
\]
\end{proof}

\section*{Acknowledgement}
We would like to thank Prof. D. Gomes for introducing us the motivation of studying mean field game system and 
several useful discussions. 
X. Su is supported by both National
Natural Science Foundation of China (Grant No. 11301513) and ``the Fundamental Research Funds for the Central Universities". 

\bibliographystyle{alpha}
\bibliography{reference}
\end{document}